\documentclass[reqno]{amsart}

\usepackage[pagewise]{lineno}

\usepackage{color}

\usepackage{hyperref}
\usepackage{comment}

\usepackage{accents}

\newtheorem{theorem}{Theorem}
\newtheorem{lemma}[theorem]{Lemma}

\newtheorem{definition}[theorem]{Definition}
\newtheorem{proposition}[theorem]{Proposition}
\newtheorem{remark}[theorem]{Remark}

\newtheorem{example}[theorem]{Example}

\usepackage{parskip}

\numberwithin{theorem}{section}
\numberwithin{equation}{section}

\begin{document}

\title[Minimal Jointly Uniform Attractor for NRDS]{Minimal jointly uniform attractor for nonautonomous random dynamical systems}

\author[P. Catuogno]{Pedro Catuogno$^1$}
\thanks{$^{1,3}$Mathematics Department, State University of Campinas (UNICAMP), Brazil. Partially supported by São Paulo Research Foundation (FAPESP), Brazil. Process Number 2020/04426-6. $^{1}$e-mail: pedrojc@unicamp.br and $^3$e-mail: ruffino@unicamp.br }

\author[A. N. Oliveira-Sousa]{Alexandre N. Oliveira-Sousa$^2$}
\thanks{$^2$Mathematics Department, Federal University of Santa Catarina (UFSC), Brazil. Partially supported by São Paulo Research Foundation (FAPESP), Process Number 2022/00176-0. e-mail: alexandre.n.o.sousa@ufsc.br}

\author[P. Ruffino]{Paulo Ruffino$^3$}



\date{}
\begin{abstract}
We introduce a notion of minimal uniform attractor for nonautonomous random dynamical systems, which depends jointly on time and on a random parameter.
Several examples are provided to illustrate the concept and to compare it with existing notions of uniform attractors in the literature. 
We further apply the abstract theory to nonautonomous random differential equations with a non-compact symbol space. In particular, we develop a method to compactify the symbol space,
by adapting techniques from the theory of deterministic nonautonomous differential equations. We also establish the stability of the minimal jointly uniform attractor by exploiting the relationship between deterministic and random dynamics.
Finally, we show that such structures arise naturally in stochastic differential equations whose noise terms carry additional time dependence, by establishing a topological conjugacy between the resulting stochastic flows and suitable random differential equations.
\end{abstract}

\maketitle

\section{Introduction}

The general theory of attractors is crucial to understand long time behaviour of any kind of dynamical systems. In particular, for nonautonomous and  random evolution equations the theory has been vastly studied and a large number of relevant problems are still open in the area. See, among many others, for instance the books \cite{Bortolan-Carvalho-Langa-book}, \cite{Chepyzhov-Vishik-book}, a large variety of of articles, among them, the classical \cite{Crauel-Flandoli-94}, more recents \cite{Cui-Langa-17}, \cite{Bortolan-Carvalho-Langa}, \cite{Caraballo-Carvalho-Oliveira-Sousa-NRA},  \cite{Fan-Su-Sun-26}, \cite{Crauel-Scheutzow-2018}, just to mention few of them whose approach are closer to this article.

In many relevant problems in the literature, some sort of uniformty of the attractors show up, in the sense that approximations to compact sets  happen  uniformly with respect to certain parameter (say, initial or final time, shift in the noise, asymptotic parameter, etc). In \cite{Chepyzhov-Vishik-book} it is studied uniform attractors for nonautonomous dynamical systems and the results are applied to evolution differential equations.
Recently, for deterministic differential equations, there were some achievememnts in the theory of uniform attractors: on the structure of this uniform attractors\cite{Carvalho-Langa-Obaya-Rocha}, attractors under perturbation \cite{Bortolan-Carvalho-Langa,Bortolan-Carvalho-Langa-book}, and results, where the driving semigroup is not invertible \cite{Carvalho-Langa-Robinson-20}. 

For some special random dynamical systems, which have a compact deterministic set which attracts uniformly with respect to the symbol space, it is possible to construct a minimal compact set that attracts uniformly, using uniform omega-limit sets, see \cite[Chap. VII]{Chepyzhov-Vishik-book}.

In \cite{Cui-Langa-17}, they study uniform attractors for nonautonomous random dynamical systems. More precisely, the authors introduce the notion of uniform attractor for a nonautonomous random dynamical system for a cocycle $\varphi$ with a product driving group acting in a symbol space $\Sigma\times \Omega$, where $\Sigma$ is the hull of a nonautonomous function and $\Omega$ is a probability space. In that situation, the uniform random attractor is a compact random set that attracts uniformly with respect to $\Sigma$. In their applications, they consider the equation
	\begin{equation}
	du=f(u) dt +g(t) dt+u\circ dW_t,
	\end{equation}
where  $\Sigma$ is the hull of $g$, so that the nonautonomous term $g(t)$ and the noise are treated separately. Our approach here considers nonautonomous terms together with the noise: 
\begin{equation}\label{int-eq-nonautonomous-noise}
	du=f(u) dt +\kappa(t)u\circ dW_t.
	\end{equation}
    
We show in Theorem \ref{Thm-conjugation-NRDS} that there is a random conjugacy between  the stochastic equation \eqref{int-eq-nonautonomous-noise} and a nonautonomous random differential equation, 
\begin{equation}\label{int-eq-NRdifferentialEq-prototype}
    \dot{u}=h(\beta(t,\theta_t\omega),u), t\geq \tau,
\end{equation}
Here, the time dependence cannot be separated from the noise, so it is not sufficient to consider only the hull of $g$: now $\beta$ is a stochastic process, i.e. for each $\omega$ fixed in $\Omega$, we have to consider the hull of the process $t\mapsto\beta(t,\theta_t\omega)$. 
In fact, we investigate the equation \eqref{int-eq-NRdifferentialEq-prototype} through two different methodologies. Firstly, we utilize the standard framework of nonautonomous random dynamical systems. Secondly, for each fixed $\omega \in \Omega$, we explore the dynamics via the skew-product semiflow associated with the equation. Specifically, we establish a relationship between the random and deterministic dynamics, which we employ here to understand of minimal uniform attractors.

 Actually, in Secton 2, in a abstract setting, we introduce the notion of a minimal uniform attractor for nonautonomous random dynamical systems, which depends jointly on time and on the random parameter. Then several examples are provided to illustrate the concept and to compare it with existing notions of uniform attractors in the literature. A theorem of existence is proved under the usual hypothesis of the existence of a family of compact absorbing sets. 
We further apply the abstract theory to nonautonomous random differential equations with a non-compact symbol space. In Section 3, we compactify the space of symbols by adapting techniques from the theory of deterministic nonautonomous differential equations. We also establish the stability of the minimal jointly uniform attractor by exploiting the relationship between deterministic and random dynamics, Theorem \ref{th-Stability}.
In Section 4, we explore particular cases of minimal jointly uniform attractors for autonomous RDS as corollaries of the nonautonomous case.
 Finally, in Section 5, we consider applications to SDE and PDEs and establish a cocycle conjugacy between a class of stochastic flows and random differential equations.

\section{Minimal joint uniform random attractors}
\label{sec:preliminaries}

This section introduces the concept of a minimal joint uniform attractor for nonautonomous random dynamical systems. We establish a theorem proving its existence, provide examples to illustrate its properties, and compare it with existing definitions of uniform attractors in stochastic or random dynamics, emphasizing its distinct advantages.

Let $(\Omega,\mathcal{F},\mathbb{P})$ a probability space endowed with  a \textbf{metric dynamical system} $\theta$, i.e., a family of measure-preserving  mappings $\{\theta_t:\Omega\rightarrow\Omega; \, t\in\mathbb{R}\}$ such that $(t,\omega)\mapsto \theta_t\omega$ is measurable, 	 $\theta_0=Id_\Omega$ and $\theta_{t+s}=\theta_t\circ \theta_s$, for all $t,s\in\mathbb{R}$. Typically $(\Omega,\mathcal{F},\mathbb{P})$ is the classical sample space of the 2-sided Wiener process with the canonical  shifts given by  $\theta_ t \, \omega (\cdot)= \omega (\cdot + t)-\omega(t)$ for all $t\in \mathbb{R}$.
	 
	Given a metric dynamical system $\{\theta_t:\Omega\rightarrow\Omega: \, t\in \mathbb{R}\}$ we define the shift operator $\Theta_t(\tau,\omega):=(t+\tau,\theta_t\omega)$ for each $(\tau,\omega)$ and and $t\in\mathbb{R}$, in the enlarged metric DS defined in  $\mathbb{R}\times\Omega$, . 
    
    Let $(X,d)$ be a complete separable metric space.
	We say that a family of maps $\{\varphi(t,\tau,\omega):X\to X; \  (t,\tau,\omega)\in\mathbb{R}^+\times\mathbb{R}\times\Omega\}$ 
	is a \textbf{nonautonomous random dynamical system} 
    (\textbf{NRDS} for short) driven by $\Theta$, which we denote simply by $(\varphi,\Theta)$, if
	\begin{enumerate}
		\item the mapping
		$\mathbb{R}^+ \times \Omega\times X\ni (t, \omega,x)\mapsto \varphi(t,\tau,\omega)x\in X$
		is measurable
		for each fixed $\tau\in\mathbb{R}$;
		\item
		$\varphi(0,\tau,\omega)=Id_X$,
		for each $(\tau,\omega)\in\mathbb{R}\times\Omega$;
		\item 
		$\varphi(t+s,\tau,\omega)=\varphi(t,\Theta_s(\tau,\omega))\, \circ \varphi(s,\tau,\omega)$,
		for every $t,s \in 
		\mathbb{R}^+ $, and $(\tau,\omega)\in\mathbb{R}\times \Omega$.
	\end{enumerate}

	We assume that the NRDS is continuous, i.e. the mapping
        $\mathbb{R}^+\times X\ni (t,x)\mapsto \varphi(t,\tau, \omega)x\in X \hbox{ is continuous, for each } (\tau, \omega)\in \mathbb{R}\times \Omega.$
        In this case, for each $(\tau,\omega)\in \mathbb{R}\times\Omega$, we obtain an \textit{evolution process} given by 
	\begin{equation*}
	\Phi_{\tau,\omega}=\{\Phi_{\tau,\omega}(t,s):=\varphi(t-s,s+\tau,\theta_s\omega)\,; \, t\geq s \},
	\end{equation*}
		i.e., the family of maps $\Phi_{\tau,\omega}$ satisfies
		\begin{enumerate}
			\item $\Phi_{\tau,\omega}(t,t)=Id_X$, for each $t\in \mathbb{R}$;
			\item $\Phi_{\tau,\omega}(t,s)\circ \Phi_{\tau,\omega}(s,r)=\Phi_{\tau,\omega}(t,r)$, for each
			$t\geq s\geq r$;
			\item the mapping $\{(t,s)\in \mathbb{R}^2: t\geq s\}\times X\ni (t,s,x)\mapsto \Phi_{\tau,\omega}(t,s)x\in X$ is continuous. 
		\end{enumerate} 

When $X$ is a Banach space, there is a well known class of examples of continuous NRDS associated to non-autonomous random differential equations of the type:
\begin{equation}\label{eq-general-NRDE-generates-NRDS}
    \dot{u}=f(\Theta_t(\tau,\omega),u), \ t\geq 0, \ u(0)=x_0\in X.
\end{equation}
where $f:\mathbb{R}\times \Omega\times X\to X$ is such that the equation is globally well-posed. For each $(\tau,\omega,x_0)\in \mathbb{R}\times \Omega\times X$, the solution of Problem \eqref{eq-general-NRDE-generates-NRDS} 
is denoted by
\begin{equation*}
    \varphi(\cdot,\tau,\omega)x_0:[0,+\infty)\to X,
\end{equation*}
which defines a NRDS $(\varphi,\Theta)$ in $X$.
In Section \ref{sec:application-SDEs}, it is shown that some non-autonomous stochastic differential equations (see Equation \eqref{eq-Stratonovich-SDE}) can be seen as a non-autonomous random differential equation as Equation \eqref{eq-general-NRDE-generates-NRDS}, due to a conjugacy of the associated dynamical systems, see Theorem \ref{Thm-conjugation-NRDS}.

The next result establishes a class of examples which reinforce the jointly dependence of both variables $(\tau,\omega)\in \mathbb{R}\times \Omega$.

\begin{proposition}\label{remark-generation-NRDS} Assume that for each
$(\tau,\omega,x_0)\in \mathbb{R}\times \Omega\times X$ the problem
    \begin{equation}\label{eq-class-NRDE-with-randomm-parameter}
\dot{v}=g(\omega,v)+h(t,\theta_t\omega,v), \ t> \tau, \ v(\tau)=x_0\in X,
\end{equation} 
has a unique solution $v(t,\tau,\omega,x_0)$ for all $t\geq \tau$. Then 
the solution flow of Equation $(\ref{eq-class-NRDE-with-randomm-parameter})$ induces a NRDS $(\varphi,\Theta)$ defined by
\begin{equation*}
\varphi(t,\tau,\omega)x_0=v(t+\tau, \tau,\theta_{-\tau}\omega,x_0), \ \forall \ t\geq 0.
\end{equation*}  
\end{proposition}

\begin{proof}
     We prove that Equation \eqref{eq-class-NRDE-with-randomm-parameter} can be written as Equation \eqref{eq-general-NRDE-generates-NRDS}, with 
     $f: \mathbb{R}\times \Omega\times X\to X$ defined by
\begin{equation}\label{eq-nonlinearity-example-g(omega)}
    f(\tau,\omega,u):=g(\theta_{-\tau}\omega,u)+h(\tau,\omega,u), \hbox{ for } (\tau,\omega,u) \in \mathbb{R}\times \Omega\times X.
\end{equation}
Let $u(t,\tau,\omega,u_0)=v(t+\tau, \tau,\omega,u_0)$, for every $t\geq 0$. 
Now, the argument follows a similar approach to \cite[Section 3]{Caraballo-Carvalho-Langa-OliveiraSousa-2021}.
We first note that $u(t, \tau,\theta_{-\tau}\omega,u_0)$ is a solution of 
\begin{equation}\dot{u}=g(\theta_{-\tau}\omega,u)+h(t+\tau,\theta_{t}\omega,u), \ t> 0, \ u(0)=x_0\in X.
\end{equation}
Since $\theta_{-(\tau+t)}\theta_t\omega=\theta_{-\tau}\omega,$ for all $t\in \mathbb{R}$, it is easy to see that $u$ satisfies
equation \eqref{eq-general-NRDE-generates-NRDS}, with  non-linearity $f$ given by \eqref{eq-nonlinearity-example-g(omega)}. 
Therefore, 
by uniqueness of solutions we have
\begin{equation*}
u(t, \tau,\theta_{-\tau}\omega,x_0)
        =v(t+\tau, \tau,\theta_{-\tau}\omega,x_0), \ \forall \ t\geq 0,
\end{equation*}  
and the proof is complete. 
\end{proof}

The technique outlined above for deriving 
represents the usual approach to generating a nonautonomous random dynamical system (NRDS), see e.g. \cite{Caraballo-Carvalho-Langa-OliveiraSousa-2021,Caraballo-Carvalho-Oliveira-Sousa-NRA,Bixiang-Wang-existence}, in Section \ref{sec-Stability-of-MJUA} we will propose a different approach.

In general, solutions of a random equations of the form \eqref{eq-class-NRDE-with-randomm-parameter}
are not cocycles under the usual metric dynamical system $\theta$. But it is indeed a cocycle in the joint metric DS defined by $\Theta$ over $\mathbb{R}\times \Omega$.


    Our goal now is to introduce a concept of \textit{minimal jointly uniform attractor} for the NRDS $(\varphi,\Theta)$, where the attraction will be uniform with respect to $\{\Theta_s(\tau,\omega): s\in \mathbb{R}\}=
\{(s,\theta_{s-\tau}\omega): s\in \mathbb{R}\}$,
for each $(\tau,\omega)\in \mathbb{R}\times \Omega$ fixed.

    Firstly, we recall measurability for multivalued functions. 
	Let $K:\Omega\to 2^X$ be a set-valued mapping with closed nonempty images. We say that $K$ is a \textbf{random set} if the mapping
	$\Omega\ni \omega \mapsto d(x,K(\omega))$ is $(\mathcal{F}, \Sigma_\mathbb{R})$-measurable for every fixed $x\in X$, see \cite[Chapter III]{CastaingValadier}.
    
\begin{definition} \label{Def: URA}
\emph{Let $\varphi$ be a NRDS, we say that a family of compact sets $\{\mathcal{A}_U(\tau,\omega): (\tau,\omega)\in \mathbb{R}\times \Omega\}$ is a \textbf{minimal jointly uniform attractor (MJUA)} if 
	\begin{enumerate}
		\item $\mathcal{A}_U$ is $\Theta$-invariant, i.e., $\mathcal{A}_U(\tau,\omega)=\mathcal{A}_U(\Theta_t(\tau,\omega))$, for each $(\tau,\omega)\in \mathbb{R}\times \Omega$;
		\item $\omega\mapsto A_U(\tau,\omega)$ is a random set, for each $\tau\in \mathbb{R}$ fixed;
		\item for each bounded set $B\subset X$, we have
		\begin{equation}
		\lim_{t\to +\infty}\sup_{s\in \mathbb{R}}\  dist(\varphi(t,\Theta_s(\tau,\omega))B,\mathcal{A}_U(\tau,\omega))=0, \ \hbox{ for all } (\tau,\omega)\in \Omega,
		\end{equation}
        where $dist(A,B)=\sup_{a\in A}\inf_{b\in B} d(a,b)$ is the Hausdorff semi-distance between $A$ and $B$;
        \item $\mathcal{A}_U(\tau,\omega)$ is the minimal closed family with Property (3).
	\end{enumerate} }
\end{definition}

The joint dependence appears from $\Theta$-invariance, i.e. 
$\mathcal{A}_U(\tau,\omega)=\mathcal{A}_U(0,\theta_{-\tau}\omega)$. 
Nevertheless, note that, starting from the initial time $\tau=0$, one can recover  $\mathcal{A}_U(\tau,\omega)$ for any $(\tau, \omega)$. For this reason, whenever convenient, we shall simplify the notation to 
\begin{equation}\label{eq-remark-joint-dependence-MJUA}
\mathcal{A}_U(\omega):=\mathcal{A}_U(0,\omega).
\end{equation}

There are many approaches of random attractors in the literature, see e.g. Arnold \cite{Arnold}, Crauel and Scheutzow \cite{Crauel-Scheutzow-2018}, Crauel and Flandoli \cite{Crauel-Flandoli-94}, among many others. As far as we know, among them, the concepts which are closely related to our Definition \ref{Def: URA} are the following: 
\begin{itemize}
    \item Chepyzhov and Vishik \cite[Chap. IV]{Chepyzhov-Vishik-book} define a deterministic minimal uniform attractor that attracts bounded subsets with respect to the initial time $s$. Definition \ref{Def: URA} extends this concept to NRDS. Indeed, for each pair $(\tau,\omega)$, the compact set 
    $\mathcal{A}_U(\tau,\omega)$ serves as a uniform attractor (with respect to $s$) for the evolution process $\{\Phi_{\tau, \omega}(t,s): t\geq s\}$; moreover, the family $\{\mathcal{A}_U(\tau,\omega):(\tau,\omega)\in \mathbb{R}\times \Omega\}$ is $\Theta$-invariant and measurable in the sense of (1) and (2) of Definition \ref{Def: URA}, respectively.
    \item Cui and Langa \cite{Cui-Langa-17} introduce the notion of uniform attractor for a nonautonomous random dynamical system for a cocycle $\varphi$ with a product driving group acting in a symbol space $\Sigma\times \Omega$, where $\Sigma$ is the hull of a nonautonomous function and $\Omega$ is the probability space. In that situation, the uniform random attractor is a compact random set that attracts uniformly with respect to $\Sigma$. In Definition \ref{Def: URA} , instead, the attraction is uniform with respect to the orbit $\{\Theta_t\omega_{\tau}: t\in \mathbb{R}\}$ of each element $\omega_{\tau}\in \mathbb{R}\times \Omega$.
	More precisely, in \cite{Cui-Langa-17} the nonautonomous term $g(t)$ and the noise are treated separately, for example, in the equation
	\begin{equation}
	du=f(u) dt +g(t) dt+u\circ dW_t,
	\end{equation}
the space of sysmbols  $\Sigma$ is the hull of $g$. Our approach here intertwine time $t$ and noise $\theta_t \omega$ such that one can consider nonautonomous diffusion terms like 
\begin{equation}
	du=f(u) dt +g(t) dt +\kappa(t)u\circ dW_t,
	\end{equation}
    as is shown is Section \ref{sec:application-SDEs}.
\end{itemize}

\medskip
\begin{example}  \label{Ex: paulo} 
\emph{
Let $\Omega = S^1$ with canonical coordinates $\omega\in [0,2\pi) \mapsto e^{i\omega}$. Take the normalized Lebesgue measure in $S^1$ and the measure preserving shift $\theta_t \omega = e^{i(t+\omega)}$. 
let $\beta:\mathbb{R}\times \Omega\to \mathbb{R}$ be a stochastic process such that 
\begin{equation*}
    \lim_{t\to +\infty}\beta(\Theta_t(\tau,\omega))=0, \ \forall \, (\tau,\omega)\in \mathbb{R}\times \Omega.
\end{equation*}
Thus, it is straightforward to verify that the nonautonomous random scalar differential equation
\begin{equation}\label{eq-first-example-R}
    \dot{x}= -x + \sin (\theta_{-\tau} \omega)+\beta(\Theta_t(\tau,\omega)), \ t\geq 0, \ x(0)= x_0,
\end{equation}
generates a NRDS $(\varphi,\Theta)$ (by Proposition \ref{remark-generation-NRDS}), where 
$\varphi(t, \tau, x_0)$ is a cocycle driven by $\Theta_t (\tau, \omega)= (\tau+t, \theta_t \omega)$. Moreover, $(\varphi,\Theta)$ has a MJUA 
$\{\mathcal{A}_U(\tau,\omega):(\tau,\omega)\in \mathbb{R}\times \Omega\}$ given by 
$$ 
\mathcal{A}_U(\tau,\omega)= \{ \sin ( \omega - \tau) \}, \ (\tau,\omega)\in \mathbb{R}\times \Omega,
$$
which naturally depends jointly on $\tau$ and on $\omega$. In fact, the space of symbols is a cylinder where the orbits of the shift $\Theta_t$ are helices. The MJUA depends on each helix the initial symbol $(\tau, \omega)$ belongs.
We remark that the previous theory of uniform attractors from \cite[Chapter VII]{Chepyzhov-Vishik-book}  applied to Example \eqref{eq-first-example-R}  provides a rougher uniform attractor, given by $[-1,1]$, which is the smallest compact subset of $\mathbb{R}$ that attracts bounded subsets uniformly with respect to $(\tau,\omega)\in \mathbb{R}\times \Omega$, in the sense of \eqref{eq-stronger-unifom-AK}. 
Similarly, uniform random attractor of Cui and Langa \cite{Cui-Langa-17} also is given by $[-1,1]$, since the uniformity is considered with respect to $\tau\in \mathbb{R}$, by taking $\Sigma=\mathbb{R}$. Therefore, the MJUA provides a more precise (a singleton) family of compact sets that attracts uniformly.  
}
\end{example}

Inspired by the deterministic ideas in \cite[Chapter IV]{Chepyzhov-Vishik-book}, we prove a result on the existence of minimal uniform attractors for NRDS.
Firstly, for this type of uniform attraction, the corresponding omega-limit set is defined as follows.
\begin{definition}
    Let $B$ be a subset of $X$. The {\bf uniform omega-limit set of $B$} is defined as 
    \begin{equation*}
        L_U(B,\tau,\omega)
        =\bigcap_{r\geq 0}\overline{\bigcup_{t\geq r}\bigcup_{s\in\mathbb{R}}\varphi(t,\Theta_s(\tau,\omega))B
        }, \ (\tau,\omega)\in \mathbb{R}\times \Omega.
    \end{equation*}
\end{definition}

    Note that, $L_U(B,\cdot)$ is $\Theta$-invariant, i.e., $L_U(B,\tau,\omega)=L_U(B,\Theta_t(\tau,\omega))$, for all $t\in \mathbb{R}$.
    This $\Theta$-invariance may not hold for forward omega-limit set of $B\subset X$ defined by 
\begin{equation*}
        L^+(B,\tau,\omega)
        =\bigcap_{r\geq 0}\overline{\bigcup_{t\geq r}\varphi(t,\tau,\omega)B
        }, \ (\tau,\omega)\in\mathbb{R}\times \Omega.
    \end{equation*}



The main result of this section assumes usual uniform asymptotically compactness, see e.g. \cite[Chapter IV, p. 84]{Chepyzhov-Vishik-book}. 

\begin{theorem}[Existence of minimal joint uniform attractor]\label{th-existence-minimal-uniform-attractor}
    Let $\varphi$ be a NRDS. Assume that there exists a family  $\{K(\tau,\omega):(\tau,\omega)\in \mathbb{R}\times \Omega\}$ of compact subsets of $X$
    such that for each bounded subset $B\subset X$ we have
\begin{equation}\label{eq-hypotheses-unifom-AK.}
		\lim_{t\to +\infty}\sup_{s\in \mathbb{R}}\  dist(\varphi(t,\Theta_s(\tau,\omega))B,K(\tau,\omega))=0.
		\end{equation}
        Then, there exists a minimal uniform attractor for $\varphi$, given by
        \begin{equation}\label{eq-th-existence-URA}
\mathcal{A}_U(\tau,\omega)=\overline{\bigcup_{B\in\mathcal{B}} L_U (B,\tau,\omega))},
        \end{equation}
        where $\mathcal{B}$ is the class of all bounded non-empty subsets of $X$.
\end{theorem}

\begin{proof}
    For each bounded subset $B$ of $X$, it is straightforward to prove that, for each $\tau,\omega\in \mathbb{R}\times \Omega$ fixed,
    $L(B,\omega)$ is a non-empty, compact set and  is the minimal closed set such that 
\begin{equation}
		\lim_{t\to +\infty}\sup_{s\in \mathbb{R}}
        dist(\varphi(t,\tau+s,\theta_s\omega)B,L_U(B,\omega))=0.
		\end{equation}
        Therefore, the family $\{\mathcal{A}_U(\tau,\omega): (\tau,\omega)\in \mathbb{R}\times \Omega\}$ defined by \eqref{eq-th-existence-URA} defines a family of compact sets which satisfies all the conditions to be a uniform attractor for $(\varphi,\Theta)$, with exception of the measurability, so let us show that
        $\omega\mapsto A_U(\tau,\omega)$ is a random set for every $\tau\in \mathbb{R}$.
        
        Let us show first that
        $\omega\mapsto L_U(B,\omega)$ is a random set, where $B$ is an open ball, we follow closely the ideas of \cite[Theorem 3.11]{Crauel-Flandoli-94}. 
        The first step is to prove that 
        $\omega\mapsto \varphi(t,s+\tau,\theta_s\omega))B$ is a random set, for every $t\geq 0$, and $s,\tau\in\mathbb{R}$.
        Second, let $x\in X$, note that
        \begin{equation*}
            d(x,\bigcup_{t\geq r}\bigcup_{s\in\mathbb{R}}\varphi(t,s+\tau,\theta_s\omega)B)=
            \inf\{d(x,\varphi(t,s+\tau,\theta_s\omega)B): t\geq r,\ s\in\mathbb{R}\}.
        \end{equation*}
        Now, given $a\in \mathbb{R}$, we have
        \begin{align*}
            \{\omega:\, \inf\{d(x,\varphi(t,s+\tau,\theta_s\omega)B):\, t\geq r,\ s\in\mathbb{R}\}<a\}=\\
            \pi_\Omega\{(t,s,\omega): \, d(x,\varphi(t,s+\tau,\theta_s\omega)B)<a, \,t\geq r, \, s\in \mathbb{R}\},
        \end{align*}
        where $\pi_\Omega:\mathbb{R}^2\times \Omega\to \Omega$ is the canonical projection onto $\Omega$. Hence, by the projection theorem \cite[Theorem III.23]{CastaingValadier}, 
        the mapping $\omega\mapsto d(x,\bigcup_{t\geq r}\bigcup_{s\in\mathbb{R}}\varphi(t,s+\tau,\theta_s\omega)B)$ is measurable with respect to the $\mathbb{P}$-completion of $\mathcal{F}$, and the proof of the assertion is complete by taking the countable intersection ($r\in \mathbb{Q}$ with $r>0$).
        Finally, to see that $\omega\mapsto \mathcal{A}_U(\tau,\omega)$ is a random set, take a countable union of balls in \eqref{eq-th-existence-URA}.
\end{proof}

    We will apply Theorem \ref{th-existence-minimal-uniform-attractor} in the examples of the next sections.
    

    \medskip
     If, instead of Condition \eqref{eq-hypotheses-unifom-AK.}, we assume that there is a deterministic compact set $K$ that 
\begin{equation}\label{eq-stronger-unifom-AK}
		\lim_{t\to +\infty}\sup_{(\tau,\omega)\in \mathbb{R}\times \Omega}\  dist(\varphi(t,\tau,\omega)B,K)=0,
		\end{equation}
        then by defining the following uniform omega-limit set
        \begin{equation*}
        L_{\mathbb{R}\times \Omega}(B)
        =\bigcap_{r\geq 0}\overline{\bigcup_{t\geq r}\bigcup_{(\tau,\omega)\in \mathbb{R}\times \Omega}\varphi(t,\tau,\omega)B
        }, \ \omega\in\mathbb{R}\times \Omega,
    \end{equation*}
    it is possible to obtain a uniform attractor $\mathcal{A}_{\mathbb{R}\times \Omega}$ with respect to $\mathbb{R}\times \Omega$, for details see \cite[Chapter VII, Thm 1.1]{Chepyzhov-Vishik-book}. 
    It is straightforward to verify that
    $\mathcal{A}_U(\tau,\omega)\subset \mathcal{A}_{\mathbb{R}\times \Omega}$, for all $(\tau,\omega)$,
with proper inclusion in certain cases (see Example \ref{Ex: paulo} ).

\begin{example} \label{Ex-MJUA for cubic}\emph{
  Let $(\Omega,\mathcal{F},\mathbb{P},\theta)$ be the canonical Wiener two-sided metric dynamical system, 
  and $a:\Omega\to (0,+\infty)$ a random variable. Consider the following scalar differential equation
\[
\dot{x}= x - \frac{x^3}{a(\omega)^2}+|g(t)z(\theta_t \omega)| x^3, \  t\geq \tau,
\]
 where $z(\theta_t \omega)= z_t(\omega)$ is the Ornstein-Uhlenbeck process and $g:\mathbb{R}\to \mathbb{R}$ such that $g(t)z(\theta_t\omega)\to 0$ as $t\to \pm \infty$, see e.g. \cite[Subsection 7.1]{Caraballo-Carvalho-Oliveira-Sousa-NRA} for examples of such functions. 
 Then, the MJUA is given by 
 \begin{equation*}
     \mathcal{A}_U(\tau,\omega)=[-a(\theta_{-\tau}\omega),a(\theta_{-\tau}\omega)], \ \hbox{ for all }(\tau,\omega)\in \mathbb{R}\times \Omega.
 \end{equation*}
 If $a$ is a bounded random variable, the uniform attractor defined by Cui and Langa \cite{Cui-Langa-17} will coincide with the uniform attractor of Chepyzhov and Vishikv \cite{Chepyzhov-Vishik-book}, given by $[-\sup a,\sup a]$. On the other hand, if $a$ is unbounded, there is no uniform attractor in the sense of \cite{Chepyzhov-Vishik-book,Cui-Langa-17}. 
}
\end{example}

Like any other definition of attractors, not every NRDS has a MJUA. Consider for example the Orstein-Uhlenbeck equation
\begin{equation}\label{eq-OU-equation}
    dz_t = -z\ dt + dW_t,
\end{equation}
which has a pullback attractor given by 
\[
z(\omega) = - \int_{-\infty}^0 e^t \omega (t)\ dt.
\]
Nevertheless, it does not have a uniform random attractor since $z(\omega)$ is Gaussian. In the next sections we establish conditions for existence of MJUA in some particular cases.

\section{Stability of Minimal Joint Uniform Attractors via Compactification of symbol space
}
\label{sec-Stability-of-MJUA}
This section presents an alternative approach to nonautonomous random differential equations by examining the hull of nonlinearities for each 
$\omega\in \Omega$ fixed.
This method allows us to use the stability of the global attractor within the skew product semi-flow to prove the stability of the minimal joint uniform attractor for the nonautonomous random dynamical systems associated. 

As before, let $X$ be a separable Banach space, and $(\Omega,\mathcal{F},\mathbb{P},\theta_t)$ be a metric dynamical system.
Consider the following differential equation
\begin{equation}\label{eq-NRDE-with-h}
\dot{v}=h(\beta(t,\theta_t\omega),v), \ t> \tau, \ v(\tau)=x_0\in X,
\end{equation}
where $h:\mathbb{R}\times X\to X$ is continuous and Lipschitz in the second variable uniformly in bounded subsets of $X$, and $\beta:\mathbb{R}\times \Omega\to \mathbb{R}$ with conditions specified below.
Throughout this section, we assume that for each $(\tau,\omega)\in \mathbb{R}\times \Omega$ and $x_0\in X$ there exists a unique solution $v(t,\tau,\omega;x_0)$, defined for all $t\geq \tau$.
As in Proposition \ref{remark-generation-NRDS}, by considering $u(t, \tau,\omega,x_0)=v(t+\tau, \tau,\omega,x_0)$, we obtain the problem 
\begin{equation}
\dot{u}=h(\beta(t+\tau,\theta_{t+\tau}\omega)),u), \ t> 0, \ u(0)=x_0\in X.
\end{equation}
Hence, if $\varphi(t,\tau,\omega)u_0$
is the solution of 
\begin{equation}
\dot{u}=h(\beta(t+\tau,\theta_{t}\omega),u), \ t> 0, \ u(0)=x_0\in X,
\end{equation}
 the solutions are associated with a NRDS 
$(\varphi,\Theta)$ given by
\begin{equation}\label{eq-generatiton-NRDS}
\varphi(t,\tau,\omega)x_0:=v(t+\tau,\tau,\theta_{-\tau}\omega;x_0), \ \forall \, t\geq 0.
\end{equation}

Our current goal is to study the asymptotic behaviour of 
\eqref{eq-generatiton-NRDS} with the standard theory of deterministic nonautonomous dynamical systems. Specifically, we will employ the theoretical framework provided in \cite[Chapter IV]{Chepyzhov-Vishik-book} to investigate the skew product semiflow associated with the hull of the mapping $$t\mapsto \beta(t+\tau,\theta_t\omega), \hbox{for each fixed } (\tau,\omega)\in \mathbb{R}\times \Omega.$$

More precisely, let $C(\mathbb{R})$ be the set of all real-valued continuous functions and define $\vartheta_tf(\cdot)=f(t+\cdot)$ the shift map in $C(\mathbb{R})$ by $t\in \mathbb{R}$. 
In $C(\mathbb{R})$ we consider $\rho$ the compact open topology, then the hull of $t\mapsto \beta(t+\tau,\theta_t\omega)$ is given by
\begin{equation}\label{eq-nonautonomous-symbol-space}
\Sigma(\tau,\omega):=\overline{\{\vartheta_t\beta(\tau+ \cdot\,,\theta_{(\cdot)} \omega): t\in\mathbb{R} \}}^\rho.
\end{equation}

    This jointly dependence, implies that $\Sigma(\tau,\omega)$ is $\Theta$-invariant in the following sense: $\Sigma(\Theta_t(\tau,\omega))= \Sigma(\tau,\omega)$ for all $t\in \mathbb{R}$. Again, similarly to \eqref{eq-remark-joint-dependence-MJUA}, when   parameter $\tau$ is omitted we mean that $\tau=0$, i.e.  $ \Sigma (\omega):= \Sigma (0,\omega)$.

\begin{remark}
\emph{
    For deterministic non-autonomous differential equations, such as 
    \begin{equation}
        \dot{u}=f(g(t),u),  \ t\geq \tau, \ v(\tau)=x_0,
    \end{equation}
    the hull of $g$, $\overline{\{\vartheta_tg(\,\cdot\,): t\in\mathbb{R} \}}^\rho$ is independent of initial time $
    \tau$. 
   However, for nonautonomous random context, the hull of $t\mapsto \beta(t,\theta_t\omega)$, inherently carries information about both the initial time $\tau$ and the realization $\omega$. 
   }
\end{remark}
We assume the following assumption on $\beta$ to ensure the compactness of its hull, which plays a crucial role in establishing existence of attractors:

\bigskip

\noindent \textbf{Condition $(H_1)$:} Assume that for each $\omega\in \Omega$
the mapping $t\mapsto \beta(t,\theta_t\omega)$ is bounded and H\"older continuous on compact intervals of $\mathbb{R}$, i.e., there exist $0<\alpha<1/2$ such that for each $M>0$ there exists $\ell=\ell(\omega,M)>0$ such that
\begin{equation}
|\beta(t,\theta_t\omega)- \beta(s,\theta_s\omega)|\leq \ell |t-s|^\alpha, \ \forall \ t,s\in [-M,M].
\end{equation}

This hypothesis is natural in stochastic analysis: It holds for any continuous semimartingale generated by Brownian motion thanks to the well-known Kolmogorov Criterion, see e.g Revuz and Yor \cite{Revuz-Yor}.

\begin{lemma}\label{lemma-sigma(omega)-compactenss}
    Assume that $\beta$ satisfies Condition \textbf{$(H_1)$}, then $\Sigma(\omega)$,
	the hull of $t\mapsto \beta(t,\theta_t\omega)$,
    is a compact set in $C(\mathbb{R})$ for each $\omega\in \Omega$. 
\end{lemma}
\begin{proof} 
 It follows directly from Corollary 2.1 of \cite[Chapter V]{Chepyzhov-Vishik-book}, where compactness of sets of functions with unbounded intervals is constructed via compactness in bounded intervals (Arzelá-Ascoli theorem). 
\end{proof}

 An alternative way to obtain compactness of $\Sigma(\omega)$ is assuming that  $\beta(t,\theta_t\omega)\to \beta_0$, as $t\to \pm \infty$ for all $\omega\in \Omega$, then 
\begin{equation}
\Sigma(\omega)=\{\vartheta_t\beta( \cdot\,,\theta_{(\cdot)} \omega): t\in\mathbb{R} \}\cup \{\beta_0^*\},
\end{equation}
where $\beta_0^*$ is the constant function equal to $\beta_0$. See \cite[Chapter 8, Section 4]{Bortolan-Carvalho-Langa-book} for more details on asymptotically autonomous problems.

	
    
\bigskip

Now, we explain how to associate Problem \eqref{eq-NRDE-with-h} with a skew-product semiflow. For each $\sigma\in \Sigma(\omega)$, we denote 
the solution of 
\begin{equation}\label{eq-semilinear-with-sigma}
\dot{u}=h(\sigma(t),u), \ t>0, \ u(0)=x_0\in X,
\end{equation}
 by $ \psi(t,\sigma)x_0:= u(t,\sigma)x_0$ for all $t\geq 0$.
Then, the mapping
\begin{equation}\label{eq-NDS-psi}
[0,+\infty)\times \Sigma(\omega) \ni (t,\sigma)\mapsto \psi(t,\sigma)\in C(X).
\end{equation}
is continuous and satisfies
\begin{enumerate}
	\item $\psi(0,\sigma)=Id_X$;
	\item $\psi(t+s,\sigma)=\psi(t,\vartheta_s\sigma)\psi(s,\sigma)$, $t+s,s\geq 0$;
\end{enumerate}
in other words $(\psi,\vartheta)$ is a non-autonomous dynamical system (NDS). Finally, 
let $\Pi(t):X\times \Sigma(\omega)\to X\times \Sigma(\omega)$ be the skew product semi-flow associate to $(\psi,\vartheta)$, i.e., for every $t\geq 0$, 
\begin{equation}
\Pi(t)(x,\sigma)=(\psi(t,\sigma)x,\vartheta_t \sigma), \ (x,\sigma)\in X\times \Sigma(\omega).
\end{equation}

\medskip
Now we establish a relation between the dynamical systems described above.
\begin{lemma} \label{Lemma: NDS and NRDS}
    Let $(\psi,\vartheta)$ be the NDS defined in \eqref{eq-NDS-psi} and $(\varphi,\Theta)$ the NRDS defined in \eqref{eq-generatiton-NRDS}.
    If  $\sigma:=\sigma_\omega=\vartheta_{\tau}\beta(\,\cdot\,,\theta_{(\cdot)}\omega)$ for some $\tau=\tau_\sigma\in \mathbb{R}$, 
   then
\begin{equation}\label{eq-relation-between-varphi-psi}
\psi(t,\sigma_\omega)=\varphi(t,\tau_\sigma,\theta_{\tau_\sigma}\omega),  
\forall \, t\in \mathbb{R}.
\end{equation}
\end{lemma}
\begin{proof}
    For $\sigma:=\sigma_\omega\in \{\beta(t+\cdot,\theta_{t+(\cdot)} \omega): t\in\mathbb{R} \}$, there exists 
$\tau_\sigma\in \mathbb{R}$ such that 
$\sigma(t)=\beta(\tau_\sigma+t,\theta_{\tau_\sigma+t}\omega)$ for every $t\in \mathbb{R}$. Then, the associate equation 
\begin{equation}
\dot{u}=h(\beta(t+\tau_\sigma,\theta_{t+\tau_\sigma}\omega),v), \ t> 0, \ u(0)=x_0,
\end{equation}
with solution given by 
$u(t+\tau_\sigma,\tau_\sigma,\omega,u_0)$. Hence, \eqref{eq-relation-between-varphi-psi} follows from \eqref{eq-generatiton-NRDS}.

\end{proof}

\medskip

\begin{lemma}\label{lemma-equivalence-A-K-between-dynamics}
    Let $(\psi,\vartheta)$ be the NDS defined in \eqref{eq-NDS-psi} and $(\varphi,\Theta)$ the NRDS defined in \eqref{eq-generatiton-NRDS}. Assume that $\Sigma(\omega)$ is compact for every $\omega\in \Omega$, then the following statements are equivalent:
    \begin{enumerate}
         \item There is a $\mathbb{K}$ compact set in $X\times \Sigma(\omega)$ that attracts bounded subsets under the action of $\Pi(\cdot)$. 
        \item There is a compact subset $K(\omega)\subset X$, such that for all bounded subset $B\subset X$, we have 
        \begin{equation*}
            \lim_{t\to +\infty}\sup_{\sigma\in \Sigma(\omega)} dist(\psi(t,\sigma)B,K(\omega))=0.
        \end{equation*}
        \item There is a family of compact sets $\{K(\omega):\omega\in \Omega\}$ of $X$ that for all bounded subset $B \subset X$, we have
        \begin{equation}\label{eq-uniformly-attraction}
            \lim_{t\to +\infty}\sup_{s\in \mathbb{R}} dist(\varphi(t,s,\theta_s\omega)B,K(\omega))=0.
        \end{equation}
    \end{enumerate}
\end{lemma}
\begin{proof}
    The equivalence between (1) and (2)
    is obtained by taking projection of $\mathbb{K}$ onto $X$, and reciprocally by defining $\mathbb{K}= K(\omega)\times \Sigma(\omega)$, see \cite[Lemma 16.5]{Carvalho-Langa-Robison-book} for details.
    The equivalence between (2) and (3) is a direct consequence of Lemma \eqref{Lemma: NDS and NRDS}.

\end{proof}

Lemma \ref{lemma-equivalence-A-K-between-dynamics} justifies that, the existence of a family of compact sets $\{K(\omega):\omega\in \Omega\}$ that attracts bounded subsets of $X$ in the sense of 
\eqref{eq-uniformly-attraction}
implies the existence of both MJUA for the NRDS $\varphi$ and a global attractor for the skew product semiflow $\Pi$.

It is well-known, see \cite[Chapter 6]{Bortolan-Carvalho-Langa}, that if 
$\{\Pi(t): t\geq 0\}$ has a global attractor $\mathbb{A}(\omega)$, then $(\psi,\vartheta)$ has a \textit{uniform attractor} given by 
\begin{equation}
\mathcal{A}_\psi(\omega)=\pi_X(\mathbb{A}(\omega)),
\end{equation}
where $\pi_X: X\times \Sigma(\omega)\to X$ is the canonical projection over $X$, i.e., $\mathcal{A}_\psi(\omega)$ satisfies
\begin{equation}\label{eq-uniform-attractor-psi}
\lim_{t\to +\infty} \sup_{\sigma\in \Sigma(\omega)} dist( \psi(t,\sigma)B,\mathcal{A}_\psi(\omega))=0,
\end{equation}
and $\mathcal{A}_\psi(\omega)$ is the minimal closed sets that satisfies \eqref{eq-uniform-attractor-psi}.

\begin{proposition}\label{proposition-equality-minimal-attractors}
    Suppose that the NRDS $(\varphi,\Theta)$ has a minimal jointly uniform attractor $\{\mathcal{A}_U(\tau,\omega):(\tau,\omega)\in \mathbb{R}\times \Omega\}$.
   Then, for every $\omega\in \Omega$, $(\psi,\vartheta)$ has a uniform attractor given by $\mathcal{A}_\psi(\omega)=\mathcal{A}_U(\omega)$.
 
\end{proposition}
\begin{proof}
    
From \eqref{eq-relation-between-varphi-psi} and \eqref{eq-uniform-attractor-psi}, for all bounded subset $B\subset X$, we have
\begin{equation}\label{eq-uniform-attractor-varphi}
\lim_{t\to +\infty} \sup_{t_\sigma\in\mathbb{R}}dist(\varphi(t,t_\sigma,\theta_{t_{\sigma}}\omega)B,\mathcal{A}_\psi(\omega))=0.
\end{equation}
Thus, by minimality, we obtain $\mathcal{A}_U(\omega)\subset\mathcal{A}_\psi(\omega)$, for all $\omega\in\Omega$.
\par Now, since $\mathcal{A}_U$ is the minimal uniform attractor for $\varphi$, Eq. \eqref{eq-relation-between-varphi-psi} implies that  for all bounded subset $B\subset X$, we have
\begin{equation}
\lim_{t\to +\infty} \sup_{\sigma\in \Sigma_0(\omega)
}
dist(\psi(t,\sigma)B,\mathcal{A}_U(\omega))=0,
\end{equation}
where $\Sigma_0(\omega):=\{\beta(t+\cdot,\theta_{t+(\cdot)} \omega): t\in\mathbb{R} \}$.
We claim that the supremum above, can be taken in $\Sigma(\omega)$. Indeed, given $\sigma \in \Sigma(\omega)$, $\epsilon>0$, and $B\subset X$ bounded, there exists $t_0=t_0(\epsilon,B)>0$ and $\{\sigma_n\}\subset \Sigma(\omega)$ such that
$\sigma_n\to \sigma$, and
\begin{equation*}
dist(\psi(t,\sigma_n)B,\mathcal{A}_U(\omega))<\frac{\epsilon}{2}, \ \forall \, t\geq t_0.
\end{equation*}
By taking $n\to +\infty$, the claim holds true by continuity of $\psi$ in $\Sigma(\omega)$. Thus, by minimality of $\mathcal{A}_\psi$, that
$\mathcal{A}_\psi(\omega)\subset \mathcal{A}_U(\omega)$, for all $\omega\in \Omega$.

\end{proof}

Next, we provide an application of Proposition \eqref{proposition-equality-minimal-attractors} to illustrate the importance of having connections between different ways to describe dynamics, as Lemma \ref{lemma-equivalence-A-K-between-dynamics}. 
\begin{example}
    Consider the nonautonomous random differential equation
\begin{equation}\label{eq-limiting-problem-variables-z}
\begin{split} 
\dot{x}_1=&
k(x_2-x_1)+x_1-\beta(t,\theta_t\omega)x_1^3,\\
\dot{x}_2=&
-k
(x_2-x_1)+x_2-\beta(t,\theta_t\omega)x_2^3,
\end{split} 
\end{equation}
for some $k>0$, and assume that $\beta$ satisfies Condition $(H_1)$. 
For each fixed $\omega\in \Omega$, consider $(\psi,\vartheta)$ the NDS associate to \eqref{eq-limiting-problem-variables-z}.
Then it is posible to apply the results from \cite{Carvalho-Langa-Obaya-Rocha}
that completly describe the dynamics of the uniform attractor $\mathcal{A}_\psi(\omega)$.
for $(\psi,\vartheta)$. 
Then, by Proposition 
\ref{proposition-equality-minimal-attractors} we obtain that the associated NRDS $(\varphi,\Theta)$ has a MJUA $\{\mathcal{A}_U(\omega):\omega\in \Omega\}$ given by  $\mathcal{A}_U(\omega)=\mathcal{A}_\psi(\omega)$, for all $\omega\in \Omega$.

\end{example}

Now, we prove the main result of this section.

\begin{theorem}[Stability of the minimal uniform attractor]
\label{th-Stability}
    Let $\varphi$ be the NRDS associated Eq. \eqref{eq-generatiton-NRDS}. 
    If $\beta$ is bounded and $\alpha$-H\"older continuous, and there is a family of compact sets
    $\{K(\omega):\omega\in \Omega\}$
    that attracts uniformly (w.r.t initial time $s$) under the action of $(\varphi,\Theta)$. Then, given $\epsilon>0$, there exists 
	 $0< \delta < \epsilon$ such that 
	\begin{equation}
	\bigcup_{t\geq 0}\bigcup_{s\in \mathbb{R}} 
	\varphi(t,s,\theta_s\omega)O_{\delta}(\mathcal{A}_U(\omega))\subset O_\epsilon(\mathcal{A}_U(\omega)).
	\end{equation}
\end{theorem}
\begin{proof}
    Note that, by Lemma \ref{lemma-equivalence-A-K-between-dynamics}, the condition that there is a family of compact sets that attracts uniformly (w.r.t. s) together with the fact that $\Sigma(\omega)$ is compact, see Lemma \ref{lemma-sigma(omega)-compactenss}, imply that
    the skew-product semiflow $\Pi(\cdot)$
    has a global attractor 
    $\mathbb{A}(\omega)\subset X\times \Sigma(\omega)$.
	Now, using that $\mathcal{A}_\psi(\omega)=\pi_X(\mathbb{A}(\omega))$, the stability of the global attractor $\mathbb{A}(\omega)$ implies that 
	for each $\epsilon>0$ there is $\delta\in (0,\epsilon)$ such that 
	\begin{equation*}
	\bigcup_{t\geq 0}\bigcup_{\sigma\in \Sigma(\tau,\omega)} 
	\psi(t,\sigma)O_{\delta}(\mathcal{A}_\psi(\tau,\omega))\subset O_\epsilon(\mathcal{A}_\psi(\tau,\omega)),
	\end{equation*}
    see \cite[Lemma 2.1]{Carvalho-Lappicy-Moreira-Sousa} for details. From Proposition \ref{proposition-equality-minimal-attractors}, we have $\mathcal{A}_U(\omega)=\mathcal{A}_\psi(\omega)$, then the proof follows from the relation given by \eqref{eq-relation-between-varphi-psi}.
    
\end{proof}

From nonautonomous deterministic differential equations, stability of the uniform attractor plays an important hole on the existence of \textit{local inertial manifolds}, which are family of subsets of $X$ that attracts exponentially bounded subsets of $X$, see \cite[Subsection 2.3]{Carvalho-Lappicy-Moreira-Sousa} for details. See Example \eqref{eq-applications-scalar-parabolic} for the special case of RDS.

\section{Minimal uniform attractor for random dynamical systems}
    In this section, we apply the results obtained before to the special case of random dynamical systems and apply the results to stochastic differential equations.

    Firstly, we recall the concept of random dynamical system in a complete separable metric space $(X,d)$. A family of maps $\{\phi(t,\omega):X\to X; \  (t,\omega)\in\mathbb{R}^+\times\Omega\}$ 
	is a continuous \textbf{random dynamical system} 
    (\textbf{RDS}) driven by a metric dynamical system $\theta$, which we denote simply by $(\phi,\theta)$, if
	\begin{enumerate}
		\item the mapping
		$\mathbb{R}^+ \times \Omega\times X\ni (t, \omega,x)\mapsto \phi(t,\omega)x\in X$
		is measurable;
		\item
		$\phi(0,\omega)=Id_X$,
		for each $\omega\in\Omega$;
		\item $\phi(t+s,\omega)=\phi(t,\theta_s\omega)\,  \phi(s,\omega)$,
		for every $t,s \in 
		\mathbb{R}^+ $, and $(\tau,\omega)\in\mathbb{R}\times \Omega$;
        \item $\phi(t,\omega):X\to X$ is a continuous map for each  $(t,\omega)\in\mathbb{R}^+\times\Omega$.
	\end{enumerate}

    Now we introduce a the \textit{minimal uniform attractor} for the RDS $(\phi,\Theta)$, where the attraction will be uniform with respect to $\{\theta_s\omega: s\in \mathbb{R}\}$ 
 for each $\omega\in \Omega$ fixed.
    
\begin{definition} \label{Def:MUA}
\emph{Let $(\phi,\theta)$ be a RDS, we say that a family of compact sets $\{\mathcal{A}_U(\omega): \omega)\in \Omega\}$ is a \textbf{minimal uniform attractor (MUA)} if 
	\begin{enumerate}
		\item $\mathcal{A}_U$ is $\theta$-invariant, i.e., $\mathcal{A}_U(\omega)=\mathcal{A}_U(\theta_t\omega)$, for each $\omega\in \Omega$;
		\item $\omega\mapsto A_U(\omega)$ is a random set;
		\item for each bounded set $B\subset X$, we have
		\begin{equation}
		\lim_{t\to +\infty}\sup_{s\in \mathbb{R}}\  dist(\varphi(t,\theta_s\omega)B,\mathcal{A}_U(\omega))=0, \ \hbox{ for all } \omega\in \Omega;
		\end{equation}
        \item $\mathcal{A}_U(\omega)$ is the minimal closed family with Property (3).
	\end{enumerate} }
\end{definition}

    With similar arguments of the proof of Theorem \ref{th-existence-minimal-uniform-attractor}, we are able to establish a condition for existence of a minimal uniform attractor for RDS. 
    
\begin{theorem}[Existence of minimal uniform attractor]\label{th-existence-minimal-uniform-attractor-RDS}
    Let $\phi$ be a RDS. Assume that there exists a family  $\{K(\omega):(\omega)\in \Omega\}$ of compact subsets of $X$
    such that for each bounded subset $B\subset X$ we have
\begin{equation}\label{eq-hypotheses-unifom.}
		\lim_{t\to +\infty}\sup_{s\in \mathbb{R}}\  dist(\varphi(t,\theta_s\omega)B,K(\omega))=0.
		\end{equation}
        Then, there exists a minimal uniform attractor for $\phi$, given by
        \begin{equation}\label{eq-th-existence-MURA}
\mathcal{A}_U(\omega)=\overline{\bigcup_{B\in\mathcal{B}} L_U (B,\omega))},
        \end{equation}
        where $\mathcal{B}$ is the class of all bounded non-empty subsets of $X$, and 
    \begin{equation*}
        L_U(B,\omega)
        =\bigcap_{r\geq 0}\overline{\bigcup_{t\geq r}\bigcup_{s\in\mathbb{R}}\phi(t,\theta_s\omega)B
        }, \ \omega\in  \Omega.
    \end{equation*}
\end{theorem}

A class of examples from Section \ref{sec-Stability-of-MJUA} are the following
\begin{equation}
\dot{v}=h(\beta(\theta_t\omega),v), \ t> 0, \ v(0)=x_0\in X,
\end{equation}
where $h:\mathbb{R}\times X\to X$ is continuous and Lipschitz in the second variable uniformly in bounded subsets of $X$, and $\beta:\Omega\to \mathbb{R}$.

\begin{lemma}
Assume that for each $\omega\in \Omega$
the mapping $t\mapsto \beta(\theta_t\omega)$ is bounded and H\"older continuous on compact intervals of $\mathbb{R}$, i.e., there exist $0<\alpha<1/2$ such that for each $M>0$ there exists $\ell=\ell(\omega,M)>0$ such that
\begin{equation}
|\beta(\theta_t\omega)- \beta(\theta_s\omega)|\leq \ell |t-s|^\alpha, \ \forall \ t,s\in [-M,M].
\end{equation}
 Then $\Sigma(\omega)$,
	the hull of $t\mapsto \beta(\theta_t\omega)$,
    \begin{equation*}
        \Sigma(\omega)=\overline{\{\beta( \theta_{t+(\cdot)} \omega): t\in\mathbb{R} \}}
    \end{equation*}
    is a compact set in $C(\mathbb{R})$ for each $\omega\in \Omega$. 
\end{lemma}

All the results of Section \ref{sec-Stability-of-MJUA} can be done for RDS, to ilustrate we state the result on the stability for MUA.

\begin{theorem}[Stability of the minimal uniform attractor]
    Let $\phi$ be the RDS associated Eq. 
    If $\beta$ is bounded and $\alpha$-H\"older continuous, and there is a family of compact sets
    $\{K(\omega):\omega\in \Omega\}$
    that attracts uniformly (w.r.t initial time $s$) under the action of $(\phi,\theta)$. Then, given $\epsilon>0$, there exists 
	 $0< \delta < \epsilon$ such that 
	\begin{equation}
	\bigcup_{t\geq 0}\bigcup_{s\in \mathbb{R}} 
	\varphi(t,\theta_s\omega)O_{\delta}(\mathcal{A}_U(\omega))\subset O_\epsilon(\mathcal{A}_U(\omega)).
	\end{equation}
\end{theorem}

An application of the above result can be explored on the following scalar parabolic PDE under random influences
\begin{equation}\label{eq-applications-scalar-parabolic}
\begin{array}{l l} 
u_t=(a_\nu(x) u_x)_x+u-\beta(\theta_t\omega) u^3, & \ x\in (0,1), \ t>0\\
u_x(t,0) = u_x(t,1)=0, & \  t>0 \\
u(0,x)=u_0(x),&  \ x\in (0,1),
\end{array} 
\end{equation}
where $\beta:\Omega\to \mathbb{R}$ satisfies Condition $(H_1)$, and 
the diffusivity 
$a_\nu \in C^2(\mathbb{R})$ 
satisfies $a_\nu(x)>0$, for all $x\in [0,1]$ and $\nu > 0$,  it is small in a neighborhood of a point $x_*\in (0,1)$ and it is large outside this neighborhood for sufficiently small $\nu>0$, see details \cite{Carvalho-Pereira-94}.
As in \cite[Subsection 3.4]{Carvalho-Lappicy-Moreira-Sousa}, we use the stability of the MUA to show that the dynamics associated with \eqref{eq-applications-scalar-parabolic} are governed by those of the ODE \eqref{eq-limiting-problem-variables-z} with time independent $\beta$, i.e., $\beta(t,\omega):=\beta(\omega)$ for all $(t,\omega)\in \mathbb{R}\times \Omega$.

\section{Applications to SDEs}
\label{sec:application-SDEs}
Finally, we study some applications to stochastic differential equations with time-dependent noise. More precisely, let $X$ be a separable Banach space. Consider a Stratonovich stochastic differential equation 
\begin{equation}\label{eq-Stratonovich-SDE}
du=Au dt + f( u) dt + \kappa(t)u\circ dW_t, \, t>\tau \ u(\tau)=x_0\in X,
\end{equation}
where $A$ generates a $C_0$-semigroup (or analytic semigroup),
$f:X\to X$ is continuous, $\kappa: \mathbb{R}\to \mathbb{R}$ is continuously differentiable
and $\{W_t:t\in \mathbb{R}\}$ is the scalar two-sided Wiener process. 

We show that the solutions of \eqref{eq-Stratonovich-SDE} generate a NRDS, by proving a topological conjugacy with the solutions of a nonautonomous random differential equation. This is extended from \cite{Duan-Lu-Schmalfuss-03}, which is done for the case $\kappa=1$.

Assume that for each $\tau$ and $u_0$, there exists a solution $u(t,\tau,\omega,x_0)$ of \eqref{eq-Stratonovich-SDE} defined for all $t\geq \tau$. 
Define 
$v(t,\tau,\omega,x_0)=e^{-\kappa(t)z(\theta_t\omega)}u(t,\tau,\omega,x_0)$, where $z$ is the Ornstein-Uhlenbeck process defined in \eqref{eq-OU-equation}. Thus, $v$ satisfies the following nonautonomous random differential equation  
\begin{equation}\label{eq-Stratonovich-NRDE}
	\dot{v} =Av+e^{-\kappa(t)z(\theta_t\omega)}
	f(e^{\kappa(t)z(\theta_t\omega)}v)
	+ [\kappa(t)-\dot{\kappa}(t)] z(\theta_t\omega)v, \ t>\tau.
\end{equation}

Then the NRDS associated to \ref{eq-Stratonovich-NRDE} is given by:
\begin{equation}\label{eq-generation-of-psi-NRDS}
\begin{aligned}
\psi(t,\tau,\omega)&
&=v(t+\tau,\tau,\theta_{-\tau}\omega, \cdot),
\end{aligned}
\end{equation}
where $v(t,\tau,\omega, v_0)$ is the solution of \eqref{eq-Stratonovich-NRDE} with initial condition $v_0=v(\tau,\tau,\omega, v_0)$ and  $t\geq 0$.
Therefore, by the reverse process, 
the candidate for solution of \eqref{eq-Stratonovich-SDE} is 
\begin{equation}\label{eq-psi-hat}
\begin{aligned} 
\hat{\psi}(t,\tau, \omega)u_0&=
e^{\kappa(t)z(\theta_{t+\tau}(\theta_{-\tau}\omega))}v(t+\tau,\tau,\theta_{-\tau}\omega,e^{\kappa(\tau)z(\theta_{\tau}(\theta_{-\tau}\omega))}x_0) \\
&=
e^{\kappa(t)z(\theta_{t}\omega)}v(t+\tau,\tau,\theta_{-\tau}\omega,e^{\kappa(\tau)z(\omega)}x_0), \ t\geq 0.
\end{aligned}
\end{equation}
By the It\^o's formula one sees that $\hat{\psi}(t,\tau,\omega)x_0$ is the solution of \eqref{eq-Stratonovich-SDE}.
	These calculations prove the following result: 
	
	\begin{theorem}[Nonautonomous cohomology]\label{Thm-conjugation-NRDS}
		Let $\psi$ and $\hat{\psi}$ be the NRDS given by equations 
		\eqref{eq-generation-of-psi-NRDS} and \ref{eq-psi-hat}, respectively. The transformation $T:\mathbb{R}\times\Omega\times X\to X$ given by 
        $$
        T(\tau,\omega,x):=e^{-\kappa(\tau)z(\omega)}x
        $$ 
        is a random cocycle conjugacy between $\psi$ and $\hat{\psi}$. Precisely, 
        for each $(\tau,x_0)\in \mathbb{R}\times X$, we have
		\begin{equation}
		\hat{\psi}(t,\tau,\omega)x_0=T^{-1}(\Theta_t(\tau,\omega),\psi(t,\tau,\omega)T(\tau,\omega)x_0), \ \forall\, t\geq 0,
		\end{equation}
		and $(\hat{\psi},\Theta)$ is a  nonautonomous random dynamical system.
	\end{theorem}


    In many papers, the conjugacy established in \cite{Duan-Lu-Schmalfuss-03} is a key ingredient to show several dynamical properties for stochastic differential equations via random differential equations, see for instance \cite{Caraballo-Langa-Robinson-98,Caraballo-Kloeden-Schmalfu,Caraballo-Han-book} for applications on existence of random pulback attractors.

Next simple example illustrates both the conjugacy and the MJUA associated:

\begin{example}\emph{ Inspired by Example \ref{Ex-MJUA for cubic} we consider the following Stratonivich stochastic differential equation
\begin{equation}\label{eq-stochact-MJUA}
    dx= \bigg(x - \frac{x^3}{a^2}\bigg )dt+\kappa(t)x\circ dW_t, \  t\geq \tau, \ x(\tau)=x_0,
\end{equation}
 where $a:\Omega\to (0,+\infty)$ is a random variable.
 Then, theorem \ref{Thm-conjugation-NRDS} guarantees that \eqref{eq-stochact-MJUA} is associated with a NRDS $(\hat{\psi},\Theta)$ which is topologically conjugated to the NRDS associated to the nonautonomous differential equation
\begin{equation}
	\dot{v} =v-e^{2\kappa(t)z(\theta_t\omega)}\frac{v^3}{a(\omega)^2}
	+ [\kappa(t)-\dot{\kappa}(t)] z(\theta_t\omega)v, \ t>\tau.
\end{equation}
For instance, if $\kappa(t)=1/t$, for all $|t|$ suitable large, then 
 by same reasoning on Example \ref{Ex-MJUA for cubic} the MJUA of $(\hat{\psi},\Theta)$ is given by 
 \begin{equation*}
     \mathcal{A}_U(\tau,\omega)=T^{-1}(\tau,\omega)[-a(\theta_{-\tau}\omega),a(\theta_{-\tau}\omega)], \ \hbox{ for all }(\tau,\omega)\in \mathbb{R}\times \Omega.
 \end{equation*} 
 }
\end{example}
In applications, to obtain hyperbolicity and topological structure stability it is usual to assume the following condition on $\kappa$:
\begin{equation}\label{eq-kappa-controls-noise}
\sup_{t\in \mathbb{R}}\{|\kappa(t)z(\theta_t\omega)|\}<\infty \ \ \hbox{ and } \ \ \ 
\sup_{t\in \mathbb{R}}\{|[\kappa(t)-\dot{\kappa}(t)]z(\theta_t\omega)|\}<\infty, \ \forall \, \omega\in \Omega.
\end{equation}
Many relevant functions $\kappa(t)$ satisfy these condition; see e.g. examples below or \cite[Remark 7.1]{Caraballo-Carvalho-Oliveira-Sousa-NRA}.

\subsection{Generalized conditions for cohomology of SDEs}

Consider the following scalar Stratonovich stochastic differential equation
\begin{equation}\label{eq-SDE-with-g(t,u)}
du=f_1(u)dt + g(t,u)\circ dW_t
\end{equation}
We want to obtain $f_1$ and $g$ such that it is possible to find a transformation $F$ between
\eqref{eq-SDE-with-g(t,u)} and
\begin{equation}\label{eq-SDE-with-k-1}
du=f_2(u)dt + k(t)u\circ dW_t,
\end{equation}
with $f_2$ to be chosen.

The transformation $F$ must satisfy 
\begin{equation}
\begin{aligned}
&\partial_uF(t,u) g(t,u)+ k(t)F(t,u)=0\\
&\partial_tF(t,u) + \partial_u F(t,u) f_2(u)=f_1(F(t,u)),
\end{aligned}
\end{equation}
where $f_2(u)=Au+f(u)$.

Then from the first equation we obtain
\begin{equation}
F(t,u)=e^{k(t)\int_{0}^{u}g(t,s)^{-1} \, ds}.
\end{equation}

Hence, we have to find $g$ such that $g(\cdot,\cdot)^{-1}$ is integrable and that
$u\mapsto \int_{0}^{u}g(t,s)^{-1}\, ds$ is invertible.
	
	This condition can obtained for instance if  $g(t,s)=[k_2(t) s^p(p+1)]^{-1}$, thus define
	\begin{equation}
	F(t,u)=e^{k(t) k_2(t) u^{p+1}}, \ \ 
	G(t,v)=\bigg(\frac{\ln(v)}{k(t) k_2(t)}\bigg)^{1/(p+1)}
	\end{equation}
	Then $F(t,\cdot)$ and $G(t,\cdot)$ are the inverses of each other.

\bibliographystyle{abbrv}
\bibliography{references_CROO}

\end{document}